\documentclass[11pt]{amsart} 

\usepackage[margin=1in]{geometry}
\usepackage{amsmath,amssymb,mathtools}
\usepackage{color}
\usepackage[usenames,dvipsnames,svgnames,table]{xcolor}
\usepackage{graphicx}
\usepackage{wrapfig}
\usepackage{enumerate}
\usepackage{multirow}
\usepackage{multicol}
\usepackage{comment}
\usepackage{caption}
\usepackage{tikz}
\usetikzlibrary{trees}
\usepackage{euscript}
\specialcomment{itsapicture}{}{}

\newcommand{\sss}{\vspace{2 mm}\noindent}

\newcommand\Tstrut{\rule{0pt}{2.6ex}}         
\newcommand\Bstrut{\rule[-1ex]{0pt}{0pt}}   
\newcommand{\ds}{\displaystyle}
\newcommand{\inv}{^{-1}}

\newcommand{\cal}{\mathcal}
\newcommand{\bbf}{\mathbb{F}}

\newcommand{\bbq}{\mathbb{Q}}

\newcommand{\bbz}{\mathbb{Z}}
\newcommand{\bbp}{\mathbb{P}}

\newcommand{\disc}{\operatorname{disc}}
\newcommand{\Gal}{\operatorname{Gal}}

\newcommand{\ord}{\operatorname{ord}}
\newcommand{\pper}{\operatorname{pper}}
\newcommand{\per}{\operatorname{per}}

\newcommand{\weight}{\operatorname{weight}}
\newcommand{\Wr}{\operatorname{Wr}}

\newtheorem{thm}{Theorem}[section]
\newtheorem{prop}[thm]{Proposition}
\newtheorem{cor}[thm]{Corollary}
\newtheorem{lem}[thm]{Lemma}

\theoremstyle{definition}

\newtheorem{defn}[thm]{Definition}

\newtheorem{example}[thm]{Example}

\theoremstyle{remark}
\newtheorem{remark}[thm]{Remark}

\allowdisplaybreaks

\begin{document}
\title{Chebyshev action on finite fields}
\author[t. alden gassert]{T. Alden Gassert}
\email{gassert@math.umass.edu}
\address{Department of Mathematics and Statistics, University of Massachusetts, Amherst, 710 N. Pleasant Street, Amherst, MA, USA 01003}

\date{September 19, 2012}

\begin{abstract}
Given a polynomial $\phi(x)$ and a finite field $\bbf_q$ one can construct a directed graph where the vertices are the values in the finite field, and emanating from each vertex is an edge joining the vertex to its image under $\phi$. When $\phi$ is a Chebyshev polynomial of prime degree, the graphs display an unusual degree of symmetry. In this paper we provide a complete description of these graphs, and then use these graphs to determine the decomposition of primes in the Chebyshev radical extensions.
\end{abstract}

\maketitle

\section{\large Introduction} 
Let $K$ be a number field and $\phi$ be a monic polynomial of degree at least 2 with coefficients in $\cal O_K$, the ring of integers of $K$. We denote the $n$-fold iterate of $\phi$ by $\phi^n(x) = \phi(\phi^{n-1}(x))$, where $\phi^0(x) := x$. For a fixed $t \in \cal O_K$, if $\phi^n(x)-t$ is irreducible for $n \ge 1$, one can obtain, very naturally, a tower of fields over $K$ in the following way. Let $\{\theta_0 = t, \theta_1, \theta_2, \ldots \}$ be a compatible sequence of preimages of $t$ satisfying $\phi(\theta_n) = \theta_{n-1}$ (and hence $\phi^n(\theta_n)-t = 0$), then we obtain a nested sequence of fields
\begin{align*}
K = K_0 \subset K_1 \subset K_2 \subset \cdots,
\end{align*}
where $K_n := K(\theta_n)$ and $[K_n \colon K] = (\deg\phi)^n$.

In this paper, we give the decomposition of prime ideals in the towers obtained when $K = \bbq$ and $\phi = T_\ell$ is a Chebyshev polynomial of the first kind of prime degree $\ell$. The number fields arising from this construction are the \emph{Chebyshev radical extensions}, and from now on, we use $K_n$ to refer to such an extension of degree $\ell^n$ over $\bbq$. 

In general, for each $d \ge 0$, $T_d \in \bbz[x]$ is the monic, degree-$d$ polynomial defined by 
\begin{align*}
T_d(z+z\inv) = z^d + z^{-d},
\end{align*}
or equivalently, $T_d(2 \cos\theta) = 2\cos(d\theta)$. These polynomials satisfy a multitude of relations (see \cite{LMT93}, Chapter 2, or \cite{Riv90}, Chapter 1), but from a dynamical standpoint, the most significant property is that these polynomials commute under composition: 
\begin{align*}
T_d \circ T_e = T_e \circ T_d = T_{de}.
\end{align*}
In particular, $T_\ell^n = T_{\ell^n}$, which provides an intimate access to each number field in the tower described above. We note that there are many values of $t\in \bbz$ for which $T_\ell^n(x)-t$ is irreducible for each $n \ge 1$. For example, if $\ell$ is an odd prime and $t$ is divisible by $\ell$ exactly once, then it can easily be shown that every iterate is Eisenstein at $\ell$. A broader result is stated in Theorem \ref{th:intro.1}(1) below.

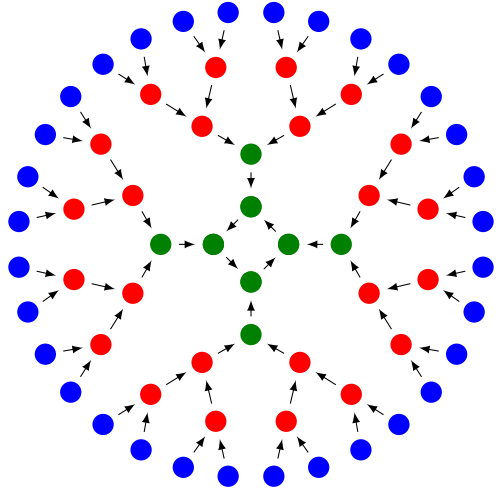
\begin{figure}[!ht]
\centering
\begin{tikzpicture}[>=latex]
	\tikzstyle{every node} = [fill,circle,outer sep = 1mm,inner sep = 1mm]
	
	\foreach \x in {0,...,3} {
		\draw[->] node[Green] (0) at (\x*90:.5) {}
			node[Green] (ahead) at (\x*90 + 90:.5) {}
			(0) edge (ahead)
			node[Green] (1) at (\x*90:1.2) {}
			(1) edge (0);
		\foreach \y in {-1,1} {
			\draw[->] node[red] (2) at (\x*90+\y*22.5:1.7) {}
				(2) edge (1);
			\foreach \z in {-1,1} {
				\draw[->] node[red] (3) at (\x*90+\y*22.5+\z*11.25:2.4) {}
					(3) edge (2);
				\foreach \w in {-1,1} {
					\draw[->] node[blue] (4) at (\x*90+\y*22.5+\z*11.25+\w*5.625:3.1) {}
						(4) edge (3);
				}
			}
		}
	}
\end{tikzpicture}
\caption{A component of the graph of $T_2$ over the finite field of order $29^4$. The color of the vertex corresponds to the smallest field containing the element associated to the vertex: green -- $\bbf_{29}$; red -- $\bbf_{29^2}$; blue -- $\bbf_{29^4}$.} \label{fig:intro.1}
\end{figure}

We determine the decomposition of primes by studying the dynamics of a Chebyshev polynomial over a finite field, an approach proposed by Aitken, Hajir, and Maire \cite{AHM05}. For a general polynomial $\phi$, the dynamics of $\phi$ over $\bbf_q$, where $q$ is a prime power, can be captured in a directed graph. The graph is constructed as follows: each element $a \in \bbf_q$ corresponds to a vertex in the graph---which by abuse of notation we also call $a$---and the graph contains the directed edge $(a,b)$ if $\phi(a) = b$. In the case of the Chebyshev polynomials, the components of the graph are radially symmetric. (See Figure \ref{fig:intro.1}.) Moreover, the number of components and their structure can be determined completely. We use the standard definitions from arithmetic dynamics to describe this structure.

\begin{defn}
Let $S$ be a set and $\phi \colon S \to S$ a map. An element $a \in S$ is \emph{preperiodic} with respect to $\phi$, and we write $\pper_{\phi,S}(a) = \rho$, if there exist minimal integers $\rho \ge 0$ and $\pi \ge 1$ such that $\phi^{\rho+\pi}(a) = \phi^\rho(a)$. Moreover, if $\rho>0$, then $a$ is \emph{strictly preperiodic}. If $\rho=0$, then $a$ is \emph{periodic} with respect to $\phi$, and we write $\per_{\phi,S}(a) = \pi$.
\end{defn}

The predictable nature of the graph allows us to deduce reducibility results for $T_\ell^n(x)-t$. For the benefit of the reader, we provide a special case of the main result (Theorem \ref{th:main}). For a prime $p$, let 
\begin{align*}
v_p = \max_{a \in \bbf_p}\{\pper_{T_\ell,\bbf_p}(a)\},
\end{align*}
and let $\overline t$ denote the reduction of $t$ modulo $p$.

\begin{thm} \label{th:intro.1}
\
\begin{enumerate}
\item If $v_p > 0$ and $\pper_{T_\ell,\bbf_p}(\,\overline t\,) = v_p$, then every iterate $T_\ell^n(x)-t$ is irreducible modulo $p$, and thus irreducible in $\bbz[x]$.
\item If $v_p > 0$ and $n \le v_p - \pper_{T_\ell,\bbf_p}(\,\overline t\,)$, then $T_\ell^n(x)-t$ splits in $\bbf_p$.
\end{enumerate}
\end{thm}

By a classical result of Dedekind, for all but finitely many primes, the factorization of the polynomial modulo $p$ and the decomposition of the ideal $p\bbz$ are linked. In particular, the factorization results define the following behavior.

\begin{cor}
Suppose $p$ does not divide the discriminant of $T_\ell^n(x)-t$. Then
\begin{enumerate}
\item $p$ is inert in $K_n$ (that is, $p\bbz$ is a prime ideal in $\cal O_{K_n}$) if $v_p > 0$ and $\pper_{T_\ell,\bbf_p}(\,\overline t\,) = v_p$;
\item $p$ splits in $K_n$ (that is, $p\bbz = \mathfrak p_1 \cdots \mathfrak p_{\ell^n}$) if $v_p > 0$ and $n \le v_p - \pper_{T_\ell, \bbf_p}(\,\overline t\,)$.
\end{enumerate}
\end{cor}

There are two cases that deserve special mention: the case $\ell$ is an odd prime and $t = 2$, and the case $\ell=2$ and $t=0$. For a generic choice of $t$, the extension $K_n$ is not Galois. Passing to the Galois closure $K_n^{\Gal}$, the Galois group $\Gal(K_n^{\Gal},\bbq)$ is non-abelian and is isomorphic to a (possibly very large) subgroup of the wreath product $\Wr(\bbz/n\bbz,S_\ell)$ (\cite{Sil07}, Theorem 3.56). In the first of the two cases listed above, the splitting field of $T_\ell^n(x)-2$ is an abelian extension of $\bbq$. In fact, the splitting field is $\bbq(\zeta_{\ell^n})^+$, the maximal totally real subfield of the cyclotomic field $\bbq(\zeta_{\ell^n})$, where $\zeta_{\ell^n}$ is a primitive $\ell^n$-th root of unity. In the latter case, the Chebyshev radical extension generated by $T_2^n(x)$ is $\bbq(\zeta_{2^{n+2}})^+$. In both cases, the decomposition of primes in the towers are known by consequence of the cyclotomic reciprocity law (\cite{Was97}, Theorem 2.13). Namely, for any prime $\ell$, the prime $p$ splits completely in $\bbq(\zeta_{\ell^n})^+$ if and only if $p$ is congruent to $\pm1$ modulo $\ell^n$. Our decomposition result provides an alternative proof of cyclotomic reciprocity in the totally real case, and more generally may be viewed as an extension of cyclotomic reciprocity to non-abelian extensions of $\bbq$.

The structure of the paper is the following. In Section 2, we give a complete description of the graph of $T_\ell$ over $\bbf_q$. We combine our knowledge of the graphs with a result by Aitken, Hajir, and Maire to prove our main theorem in Section 3. The connection to cyclotomic reciprocity is also presented in this section. In Section 4, we answer a question posed by Jones regarding the density of periodic points as the order of the field goes to infinity.

\section{\large Description of Graphs}
The study of maps over finite fields has a long history. \emph{Permutation polynomials}---polynomials giving a permutation of $\bbf_q$---are of particular interest due to their number theoretic properties and cryptographic applications. The \emph{Dickson polynomials} provide the classic examples. The Dickson polynomial $D_{d,a}$ is the unique, degree-$d$ polynomial satisfying the identity
\begin{align*}
D_{d,a}(z+az\inv) = z^d + \left(\frac{a}{z}\right)^d,
\end{align*}
and from the definitions, it is clear that $D_{d,1}(x)=T_d(x)$ and $D_{d,0}(x) = x^d$. It is known that for $a \in \bbf_q^\times$, the polynomial $D_{d,a}$ permutes $\bbf_q$ if and only if $\gcd(d,q^2-1) = 1$, and $D_{d,0}$ permutes $\bbf_q$ if and only if $\gcd(d,q-1) = 1$ (\cite{LMT93}, Theorem 3.2 and Theorem 3.1, respectively). Furthermore, the cycle structure of $D_{d,a}$ over $\bbf_q$ for $a \in \{-1, 0, 1\}$ is given in Lidl and Mullen \cite{LM91}. For more on permutation polynomials, see Section 7 of Lidl and Niederreiter \cite{LN97}.

As our previous discussion indicated, we require a theory that includes a description of the vertices that are not contained in cycles. For the remainder of the paper, we let $\ell$ be a prime number, $p$ a prime different from $\ell$, and we define $G(\ell,p,n)$ to be the graph of $T_\ell$ over $\bbf_{p^n}$. Since the graph is a direct representation of the orbits of elements in the finite field, the terms from dynamics naturally extend to describe vertices in the graph and vice versa. For example, `periodic' and `contained in a cycle' are synonymous in this setting. Likewise, `preperiod' is the same as the notion `distance to the nearest cycle'.

\subsection{Structure of $G(\ell,p,n)$}
For any polynomial over any finite field, the components of the graph have the same basic description: every component contains exactly one cycle, and the preperiodic vertices (if any) are arranged into trees branching out from the cycle. For a generic map, the branching of the trees varies greatly making the symmetric branching in the case of the Chebyshev polynomials all the more intriguing. The symmetric arrangement of the preperiodic vertices for the maps $T_2$ and $x^2$ over fields of prime order was studied by Vasiga and Shallit \cite{VS04}. Using similar methods, we generalize their results to the graphs $G(\ell,p,n)$.

For the remainder of this section we only consider the map $T_\ell$ over $\bbf_{p^n}$, henceforth $\pper(a)$ and $\per(a)$ will be understood to signify the preperiod and period, respectively, of an element $a \in \bbf_{p^n}$ with respect to $T_\ell$, unless otherwise noted. Let
\begin{align*}
c(d) = \ord_{(\bbz/d\bbz)^\times/(\pm 1)}\ell.
\end{align*}

\begin{prop} \label{prop:preperiod/period calculation}
Let $a \in \bbf_{p^n}$ and $\alpha \in \bbf_{p^{2n}}^\times$ be a root of $u(x)$, where $u(x) = x^2-ax+1$ (hence $a = \alpha + \alpha\inv$). We write $\ord_{\bbf_{p^{2n}}^\times}\alpha = \ell^k d$ where $\gcd(\ell,d) =1$. Then $\pper(a) = k$ and $\per(T_\ell^k(a)) = c(d)$.
\end{prop}

\begin{proof}
Suppose that $\pper(a) = \rho$, and $\per(T_\ell^\rho(a)) = \pi$. Then $a$ satisfies the relation $T^{\pi+\rho}_\ell (a) = T_\ell^\rho(a)$, which leads to the following equivalences.
\begin{align*}
T^{\pi+\rho}_\ell (a) = T^\rho_\ell (a) \quad&\Leftrightarrow\quad \alpha^{\ell^{\pi+\rho}} + \alpha^{-\ell^{\pi+\rho}} = \alpha^{\ell^\rho}+\alpha^{-\ell^\rho}\\
&\Leftrightarrow\quad \alpha^{\ell^{\pi+\rho}}\alpha^{\ell^{\pi+\rho}} -\alpha^{\ell^\rho}\alpha^{\ell^{\pi+\rho}} - \alpha^{-\ell^\rho}\alpha^{\ell^{\pi+\rho}}+ 1= 0\\
&\Leftrightarrow\quad \left(\alpha^{\ell^{\pi+\rho}} - \alpha^{\ell^\rho}\right)\left(\alpha^{\ell^{\pi+\rho}} - \alpha^{-\ell^\rho}\right)= 0\\
&\Leftrightarrow\quad \alpha^{\ell^{\pi+\rho}} = \alpha^{\ell^\rho} \quad \text{ or } \quad \alpha^{\ell^{\pi+\rho}} = \alpha^{-\ell^\rho}\\
&\Leftrightarrow\quad \alpha^{\ell^\rho(\ell^\pi-1)} = 1 \quad \text{ or }\quad \alpha^{\ell^\rho(\ell^\pi+1)} = 1\\
&\Leftrightarrow\quad \ell^k d \mid \ell^\rho(\ell^\pi - 1) \quad \text{ or } \quad \ell^k d \mid \ell^\rho(\ell^\pi + 1)
\end{align*}
which implies that $k \mid \rho$ and $\ell^\pi \equiv \pm 1 \pmod d$. Due to the minimality of $\pi$ and $\rho$, it follows that $\rho = k$ and $\pi = c(d)$.
\end{proof}

Note that if $u$ is irreducible, the roots of $u$ are permuted by the Frobenius endomorphism. That is, $\alpha^{p^n} = \alpha\inv$, hence $\alpha$ is an element of the cyclic subgroup of order $p^n+1$, which we denote by $H$. Otherwise $u$ is reducible, so $\alpha \in \bbf_{p^n}^\times$. In fact, the map $x \mapsto x+x\inv$ gives a 2-to-1 correspondence between the elements of these cyclic subgroups and elements of $\bbf_{p^n}$ whenever $\alpha \neq \alpha\inv$, which is precisely when $\alpha \not\in\{-1,1\}$. For $\alpha \in \bbf_{p^n}^\times$, this correspondence is clear, and for $\alpha \in H \setminus\{-1,1\}$, the map is the trace of $\alpha$. Thus the orbits of elements of $\bbf_{p^n}$ are completely classified by the orders of elements in $\bbf_{p^n}^\times$ and $H$. Namely, for each divisor $d>2$ of $p^n-1$, there are $\varphi(d)$ elements of $\bbf_{p^n}^\times$ of order $d$, which in turn correspond to $\varphi(d)/2$ elements of $\bbf_{p^n}$. Likewise, for every divisor $d>2$ of $p^n+1$, there are $\varphi(d)$ elements of $H$ of order $d$, which correspond to $\varphi(d)/2$ elements of $\bbf_{p^n}$. ($\varphi$ denotes the Euler totient function.) The divisors 2 and 1 correspond to $\alpha = -1$ and $\alpha = 1$, respectively, which yield $-2$ and $2$ in $\bbf_{p^n}$. The preperiod of $a \in \bbf_{p^n}$ is the $\ell$-valuation of the order of the corresponding $\alpha$, so to distinguish the $\ell$ part, we write
\begin{align*}
p^n-1 = \ell^{\lambda^-}\omega^- \quad \text{ and } \quad p^n+1 = \ell^{\lambda^+}\omega^+
\end{align*}
where $\gcd(\ell,\omega^-) = \gcd(\ell,\omega^+) = 1$. Additionally, we introduce the following terms to describe the trees of preperiodic vertices.

\begin{defn}
A \emph{rooted tree} is a directed tree graph where all the edges are oriented towards a specified vertex, the \emph{root}. The \emph{indegree} of a vertex is the number of incoming edges to that vertex, and a vertex is a \emph{leaf} if its indegree is zero. The \emph{height} of a vertex is the distance from that vertex to the root, and the \emph{height} of a rooted tree is the length of the longest path in the tree. Necessarily, this path begins at a leaf and terminates at the root. A rooted tree is \emph{$n$-ary} if the indegree of every non--leaf is $n$, and a tree is \emph{complete} $n$-ary if all leaves have the same height.
\end{defn}

\begin{thm} \label{th:structure of graph}
\
\begin{enumerate}
\item For each divisor $d > 2$ of $\omega^-$, the graph $G(\ell,p,n)$ contains $\varphi(d)/(2c(d))$ cycles of length $c(d)$. If $\lambda^-\ge 1$, every vertex in these cycles is adjacent to $\ell-1$ strictly preperiodic vertices. Each of these preperiodic vertices is the root of a complete $\ell$-ary tree of height $\lambda^--1$.

\item For each divisor $d > 2$ of $\omega^+$, the graph $G(\ell,p,n)$ contains $\varphi(d)/(2c(d))$ cycles of length $c(d)$. If $\lambda^+\ge 1$, every vertex in these cycles is adjacent to $\ell-1$ strictly preperiodic vertices. Each of these preperiodic vertices is the root of a complete $\ell$-ary tree of height $\lambda^+-1$.

\item If $\ell$ is an odd prime, the vertices $2$ and $-2$ are fixed. If $\max\{\lambda^-, \lambda^+\} \ge 1$, each of these vertices is adjacent to $(\ell-1)/2$ strictly preperiodic vertices. Each of these preperiodic vertices is the root of a complete $\ell$-ary tree of height $\max\{\lambda^-, \lambda^+\} -1$.

\item If $\ell=2$, then $G(\ell,p,n)$ contains the edges $(2,2), (-2,2)$, and $(0,-2)$. Moreover, 0 is the root of a complete 2-ary (binary) tree of height $\max\{\lambda^-, \lambda^+\}-2$.
\end{enumerate}
\end{thm}

\begin{proof} 
The argument for all four cases is essentially the same, so we will prove (1), and simply remark on the other cases. Fix a divisor $d>2$ of $\omega^-$. By Proposition \ref{prop:preperiod/period calculation} and the previous discussion, $G(\ell,p,n)$ contains $\varphi(d)/2$ vertices of period $c(d)$. If $\lambda^- = 0$, we are done. Otherwise, the vertices of preperiod $k$ that attach to these cycles correspond to the divisor $d\ell^k$ of $p^n-1$, where $1 \le k \le \lambda^-$. There are
\begin{align*}
\frac{\varphi(d\ell^k)}{2} = \frac{\varphi(d)}{2}(\ell-1)\ell^{k-1}
\end{align*}
such vertices. Certainly, the vertices of preperiod $\lambda^-$ have indegree 0, and since the indegree of any vertex cannot exceed $\ell$, it follows that the indegree of all other vertices in these components is $\ell$. This concludes the proof of (1), and the argument for (2) is identical. Cases (3) and (4) involve special values for the Chebyshev polynomials. It is known that when $\ell$ is odd, $\pm2$ is a fixed point of $T_\ell(x)\mp2$ and all other roots have multiplicity 2. Likewise, 0 is a multiple root of $T_2(x)+2$.
\end{proof}

Furthermore, we have a precise count for the number of vertices of a given preperiod.

\begin{thm} \label{th:point count}
Let $\lambda_m = \max\{\lambda^-, \lambda^+\}$, and if $\lambda_m > 0$, let $\omega_m$ be the integer for which $\ell^{\lambda_m}\omega_m \in \{p^n-1, p^n+1\}$.

\begin{enumerate}
\item For $\ell$ odd, the number of vertices of preperiod $k$ in $G(\ell,p,n)$ is given by
\begin{center}
\renewcommand{\arraystretch}{2}
\begin{tabular}{|l|c|c|}
\hline
Preperiod & $k=0$ & $1 \le k \le \lambda_m$\\
\hline
\# of points & $\ds\frac{\omega^- + \omega^+}{2}$ & $(\ell-1)\ell^{k-1}\ds\frac{\omega_m}{2}$\\[1.5mm]
\hline
\end{tabular}
\end{center}

\item For $\ell=2$, the number of vertices of preperiod $k$ in $G(2,p,n)$ is given by
\begin{center}
\renewcommand{\arraystretch}{2}
\begin{tabular}{|l|c|c|c|}
\hline
Preperiod & $k=0$ & $k=1$ & $2 \le k \le \lambda_m$\\
\hline
\# of points & $\ds\frac{\omega^- + \omega^+}{2}$ & $\ds\frac{\omega^- + \omega^+}{2}$ & $2^{k-2}\,\omega_m$\\[1.5mm]
\hline
\end{tabular}
\end{center}
\end{enumerate}
\end{thm}

\begin{proof}
When $\ell=2$, both $\omega^-$ and $\omega^+$ are odd. Thus the number of periodic vertices is
\begin{align*}
1+\sum_{\substack{ d \,\mid\, \omega^- \\ d\, >\, 2}} \frac{\varphi(d)}{2} + \sum_{\substack{ d \,\mid\, \omega^+ \\ d\, >\, 2}}\frac{\varphi(d)}{2} = 1+\frac{\omega^- -1}{2} + \frac{\omega^+ -1}{2} = \frac{\omega^- + \omega^+}{2}.
\end{align*}
Here, $\min\{\lambda^-, \lambda^+\}=1$, hence by Theorem \ref{th:structure of graph} there is one vertex of preperiod 1 adjacent to each periodic vertex. For all other preperiods, the number of vertices is
\begin{align*}
\frac{\varphi(2^k)}{2} + \sum_{\substack{ d \,\mid\, \omega_m \\ d\, >\, 2}} \frac{\varphi(2^kd)}{2} = 2^{k-2} + 2^{k-2}(\omega_m-1) = 2^{k-2}\omega_m.
\end{align*}
When $\ell$ is odd, both $\omega^-$ and $\omega^+$ are even. Hence the number of periodic vertices is
\begin{align*}
2+\sum_{\substack{ d \,\mid\, \omega^- \\ d\, >\, 2}} \frac{\varphi(d)}{2} + \sum_{\substack{ d \,\mid\, \omega^+ \\ d\, >\, 2}}\frac{\varphi(d)}{2} = 2+\frac{\omega^- -2}{2} + \frac{\omega^+ -2}{2} = \frac{\omega^- + \omega^+}{2}.
\end{align*}
In this case, $\min\{\lambda^-, \lambda^+\} = 0$, so for all other preperiods, the number of vertices is
\begin{align*}
\varphi(\ell^k) + \sum_{\substack{ d \,\mid\, \omega_m \\ d\, >\, 2}} \frac{\varphi(d\ell^k)}{2} = \varphi(\ell^k) + \varphi(\ell^k)\frac{\omega_m-2}{2} = (\ell-1)\ell^{k-1}\frac{\omega_m}{2}.\end{align*}
\end{proof}

\begin{example}
Using Theorem \ref{th:structure of graph} we determine the structure of the graph of $T_3$ over $\bbf_{53}$ by considering the divisors of 52 and 54 ($p-1$ and $p+1$, respectively). The graph is shown in Figure \ref{fig:G(3,53,1)}. When the degree of the Chebyshev polynomial is odd, the values $-2$, 0, and 2 are fixed, and we have labeled these values on the graph. Double arrows indicate double roots.
\end{example}

\begin{figure}[!ht]
\begin{tabular}{|c|c|c|c|c|}
\hline
\multicolumn{5}{|c|}{$\ell=3, p = 53, n=1$} \\
\hline\hline
Divisors of 52 & Number of points in $\bbf_{53}$ & Period & Preperiod & Cycles of this type\\ \hline
4 & 1 & 1 & 0 & 1\\
\hline 
13& 6 & 3 & 0 & 2\\
\hline
26 & 6 & 3 & 0 & 2\\
\hline
52 & 12 & 6 & 0 & 2\\
\hline
\hline
Divisors 54 &&&&\\
\hline
1 & 1 & 1 & 0 & \multirow{4}{*}1\\
3 & 1 & - & 1 & \\
9 & 3 & - & 2 & \\
27 & 9 & - & 3 & \\
\hline
2 & 1 & 1 & 0 & \multirow{4}{*}1\\
6 & 1 & - & 1 & \\
18 & 3 & - & 2 & \\
54 & 9 & - & 3 & \\
\hline
\end{tabular}

\vspace{3mm}

\begin{tikzpicture}[>=latex, grow = up,
every node/.style={circle,fill=Green,outer sep = 1mm}]
\node[label=left:2] at (0,0) (2) {}
	child[<<-,rotate=90,level distance=10mm,grow cyclic,level 2/.style ={<-,sibling angle=60},level 3/.style = {sibling angle=40}] { node {}
		child{ node{} 
			child{ node{}}
			child{ node{}}
			child{ node{}}
			}
		child{ node{} 
			child{ node{}}
			child{ node{}}
			child{ node{}}
			}
		child{ node{} 
			child{ node{}}
			child{ node{}}
			child{ node{}}
			}
		};
\draw (2) edge[loop below] (2);
\node[label=left:$-2$] at (5,0) (-2) {}
	child[<<-,rotate=90,level distance=10mm,grow cyclic,level 2/.style ={<-,sibling angle=60},level 3/.style = {sibling angle=40}] { node {}
		child{ node{} 
			child{ node{}}
			child{ node{}}
			child{ node{}}
			}
		child{ node{} 
			child{ node{}}
			child{ node{}}
			child{ node{}}
			}
		child{ node{} 
			child{ node{}}
			child{ node{}}
			child{ node{}}
			}
		};
\draw (-2) edge[loop below] (-2);
\node[label=left:0] at (2.5,0) (0) {};
\draw (0) edge[loop below] (0);

\draw (-3,1.5) node (1) {} ++(240:1) node (2) {}  ++(180:1) node (3) {} ++(120:1) node (4) {} ++(60:1) node (5){} ++(0:1) node(6){};
\draw[->] (1) edge (2) 
	(2) edge (3)
	(3) edge (4)
	(4) edge (5)
	(5) edge (6)
	(6) edge (1);
	
\draw(-2,0) node (1) {} ++(180:1) node (2) {} ++(60:1) node (3) {};
\draw[->] (1) edge (2) (2) edge (3) (3) edge (1);
\draw(-2,3) node (1) {} ++(180:1) node (2) {} ++(300:1) node (3) {};
\draw[->] (1) edge (2) (2) edge (3) (3) edge (1);

\draw (8,1.5) node (1) {} ++(300:1) node (2) {}  ++(0:1) node (3) {} ++(60:1) node (4) {} ++(120:1) node (5){} ++(180:1) node(6){};
\draw[->] (1) edge (2) 
	(2) edge (3)
	(3) edge (4)
	(4) edge (5)
	(5) edge (6)
	(6) edge (1);
	
\draw(7,0) node (1) {} ++(0:1) node (2) {} ++(120:1) node (3) {};
\draw[->] (1) edge (2) (2) edge (3) (3) edge (1);
\draw(7,3) node (1) {} ++(0:1) node (2) {} ++(240:1) node (3) {};
\draw[->] (1) edge (2) (2) edge (3) (3) edge (1);
\end{tikzpicture}
\caption{$G(3,53,1)$.} \label{fig:G(3,53,1)}
\end{figure}
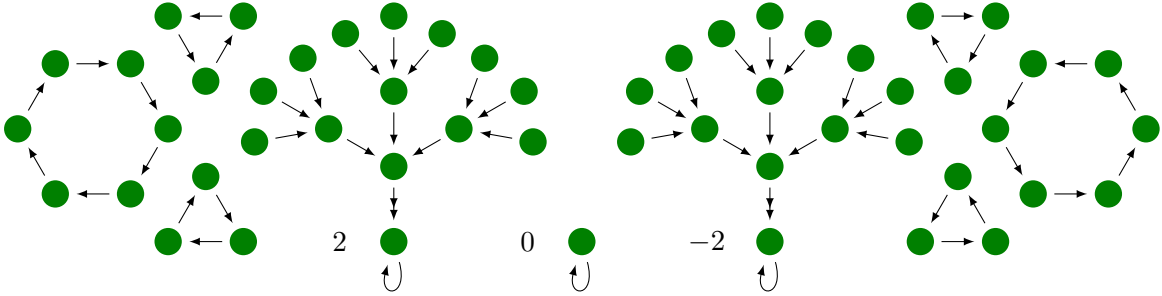

\subsection{Weights}
Each vertex also carries with it a \emph{weight}, which is its exact degree over $\bbf_p$. That is, for $a \in G(\ell,p,n)$, 
\begin{align*}
\weight(a) = [\bbf_p(a)\colon \bbf_p].
\end{align*}
It follows immediately that $G(\ell,p,\weight(a))$ is the smallest subgraph of this type containing $a$. Furthermore, $\weight(a)$ is also the smallest power for which 
\begin{align*}
(\ord_{\bbf_{p^{2n}}^\times}\alpha )\mid p^{\weight(a)}-1 \quad \text{ or } \quad (\ord_{\bbf_{p^{2n}}^\times}\alpha )\mid p^{\weight(a)}+1,
\end{align*}
where $a = \alpha + \alpha\inv$ as before. In Section 3, the weights of the preperiodic vertices will play an important role, so we are particularly interested in the $\ell$-valuation of $p^{2n}-1$. We compute this valuation now. 

\begin{lem}\label{lem:valuation} 
Let $\mu = \ord_{(\bbz/\ell\bbz)^\times/(\pm1)}p$, and $\nu_\ell$ denote the usual $\ell$-valuation.
\begin{align*}
\nu_\ell(p^{2n}-1) = 
\begin{cases}
\nu_\ell(p^{2\mu}-1)+\nu_\ell(n) & \text{ if } \mu \mid n,\\
0 & \text{ otherwise.}
\end{cases}
\end{align*}
\end{lem}

\begin{proof} 
Note that $p^{2n}-1 = (p^n+1)(p^n-1)$ is divisible by $\ell$ if and only if $p^n \equiv \pm1 \pmod \ell$, if and only if $\mu$ divides $n$. Suppose then that $\mu$ divides $n$ and write $n = k\mu$ for some integer $k$. Then 
\begin{align*}
p^{2n}-1 = p^{2k\mu}-1 = (p^{2\mu})^k-1 = (p^{2\mu}-1)\big((p^{2\mu})^{k-1} + (p^{2\mu})^{k-2} + \cdots +(p^{2\mu})^0\big).
\end{align*}
Since $p^{2\mu}\equiv 1 \pmod\ell$, we see that $\nu_\ell(p^{2n}-1) = \nu_\ell(p^{2\mu}-1) + \nu_\ell(k)$. Moreover, $\nu_\ell(\mu)=0$ since $\mu$ is a divisor of $\ell-1$. Thus $\nu_\ell(k) = \nu_\ell(k\mu) = \nu_\ell(n)$, completing the proof.
\end{proof}

One of the consequences of this result is the following. For the remainder of the paper, let $D_1, D_2 \in \{p^\mu-1,p^\mu+1\}$ be the number with the larger and smaller $\ell$-valuation, respectively, and fix a cycle $C$ in $G(\ell,p,1)$. The vertices in $C$ correspond to a divisor $d$ of $D_1$ or $D_2$, and for now, assume that $d$ divides $D_1$. Then in $G(\ell,p,\mu)$, the trees of preperiodic vertices in component containing $C$ have height $\nu_\ell(D_1)$. Furthermore, the weight of each preperiodic vertex is $\mu$. Moreover, in $G(\ell,p,\mu\ell^k)$ for $k \ge 1$, the trees of preperiodic vertices in component containing $C$ have height $\nu_\ell(D_1) + k$. It follows by induction on $k$, the vertices of preperiod $\nu_\ell(D_1)+k$ have weight $\mu\ell^k$. So what we see is that for the cycles in $G(\ell,p,1)$ corresponding to divisors of $D_1$, there is an initial spurt in the heights of the trees when moving from $\bbf_p$ to $\bbf_{p^\mu}$. From then on, moving from the field of order $p^{\mu\ell^k}$ to the field of order $p^{\mu\ell^{k+1}}$ only increases the heights of the trees by 1.

On the other hand, for cycles corresponding to divisors of $D_2$, any large jump in the height of the trees will come in the extension $\bbf_{p^{2\mu}}$, since $\nu_\ell(p^{2\mu}-1) = \nu_\ell(D_1D_2)$ and $\nu_\ell(D_1)$ may be large. Note that $\nu_\ell(D_2) = 0$ if $\ell$ is odd, otherwise $\nu_2(D_2)=1$. Any preperiodic vertices added to this component when moving from $\bbf_p$ to $\bbf_{p^{2\mu}}$ have weight $2\mu$. Following that, moving from $\bbf_{p^{2\mu\ell^k}}$ to $\bbf_{p^{2\mu\ell^{k+1}}}$ increases the height of the trees by 1. We have proven the following.

\begin{thm} \label{th:weights}
Let $a \in G(\ell,p,2\mu\ell^n)$ and assume that $a$ is strictly preperiodic. The weight of $a$ is determined by $\rho=\pper_{T_\ell,\bbf_{p^{2\mu\ell^n}}}(a)$ as follows.
\begin{enumerate}
\item If $\ell$ is odd and $a$ is connected to a cycle corresponding to a divisor of $D_1$,
\begin{align*}
\weight(a) = 
\begin{cases}
\mu & \text{ if } 0 < \rho \le \nu_\ell(D_1),\\
\mu\ell^k & \text{ if } \rho = \nu_\ell(D_1) + k \text{ for } k \in \{1,\ldots,n\}.
\end{cases}
\end{align*}

\item If $\ell$ is odd and $a$ is connected to a cycle corresponding to a divisor of $D_2$,
\begin{align*}
\weight(a) = 
\begin{cases}
2\mu & \text{ if } 0 < \rho \le \nu_\ell(D_1),\\
2\mu\ell^k & \text{ if } \rho = \nu_\ell(D_1) + k \text{ for } k \in \{1,\ldots,n\}.
\end{cases}
\end{align*}
\item If $\ell=2$ and $a$ is connected to a cycle corresponding to a divisor of $D_2$,
\begin{align*}
\weight(a) = 
\begin{cases}
1 & \text{ if } \rho = 1,\\
2 & \text{ if } 1 < \rho \le \nu_2(D_1),\\
2^k & \text{ if } \rho = \nu_2(D_1) + k \text{ for } k \in \{1,\ldots,n-1\}.
\end{cases}
\end{align*}
\end{enumerate}
\end{thm}

\begin{example}
The weight of a preperiodic vertex depends on its preperiod and whether the cycle its component corresponds to a divisor of $D_1$ or $D_2$. Returning to the previous example where $\ell=3$ and $p=53$, we have $\mu = 1$, $D_1 = 54$ and $D_2 = 52$. As we saw in Figure \ref{fig:G(3,53,1)}, the cycles corresponding to divisors of $54$ have trees of height 3, and the cycles corresponding to divisors of $52$ have trees of height 0. Over the finite field of order $53^{18}$, all cycles have trees of height 5. A table of weights is given in Figure \ref{table:weights}. Also see Figure \ref{fig:weights}.
\end{example}

\begin{figure}[!ht]
\centering
\begin{tabular}{|c|c|c|}
\hline
Preperiod & Divisor of $D_1$ & Divisor of $D_2$ \Tstrut\Bstrut\\
\hline\hline
0&1&1\\
1&1&2\\
2&1&2\\
3&1&2\\
4&3&6\\
5&9&18\\
\hline
\end{tabular}
\caption{Table of weights for $G(3,53,18)$.}\label{table:weights}
\end{figure}

\begin{example}
The graph $T_2(x) = x^2-2$ over $\bbf_{3^4}$ is shown in Figure \ref{fig:G(2,3,4)}.
\end{example}

\begin{figure}[!ht]
\begin{center}
\begin{tikzpicture}[>=latex,every node/.style={circle,fill=Green,outer sep = 1mm}]
\draw[label distance = 2mm] (0,0) node (one) {}
	++(290:1) node[label = below:2] (two) {}
	(two) edge[loop below] (two)
	(one) edge[->>] (two);
\foreach \x in {-1,0,1} {
	\draw[->] (180+60*\x:1) node (0) {}
		(0) -- (one);
		\foreach \y in {-1,0,1} {
			\draw[->] (180+60*\x+20*\y:2) node (1) {}
				(1) -- (0);
			\foreach \z in {-1,0,1} {
				\draw[->] (180+60*\x+20*\y+20*\z/3:3) node[fill=red,inner sep = 1mm] (2) {}
					(2) -- (1);
				\foreach \w in {-1,0,1} {
					\draw[->] (180+60*\x+20*\y+20*\z/3+20*\w/9:4) node[fill=blue,inner sep = .5mm] (3) {}
						(3) -- (2);
				}
			}
		}
	}
	\draw (5,0) node[fill=none] (a) {};
	\foreach \x in {-1,1} {
		\draw[label distance = 2mm,->] (a)++(90+\x*90:.5) node[Brown] (0) {}
			(a)+(-90:.75) node[Green,label=below:0] (1) {} 
			(0) -- (1)
			(1) edge[loop below] (1);
		\foreach \y in {-1,0,1} {
			\draw[->] (a)++(90+\x*90+\y*60:1.4) node[Brown] (1) {}
				(1) -- (0);
			\foreach \z in {-1,0,1} {
				\draw[->] (a)++(90+\x*90+\y*60+\z*20:2.3) node[Brown] (2) {}
					(2) -- (1);
				\foreach \w in {-1,0,1} {
					\draw[->] (a)++(90+\x*90+\y*60+\z*20+\w*20/3:3.2) node[Violet,inner sep = 1mm] (3) {}
						(3) -- (2);
					\foreach \p in {-1,0,1} {
						\draw[->] (a)++(90+\x*90+\y*60+\z*20+\w*20/3+\p*20/9:4.1) node[Navy,inner sep = .5mm] (4) {}
							(4) -- (3);
					}
				}
			}
		}
	}
	\draw (10,-2.5) node[fill=none,Green] (f) {\Large$\bbf_5$}
			++(0,2) node[fill=none,red] (f3) {\Large$\bbf_{5^3}$}
			++(0,2) node[fill=none,blue] (f9) {\Large$\bbf_{5^9}$}
			++(1.5,-3) node[fill=none,Brown] (f2) {\Large$\bbf_{5^2}$}
			++(0,2) node[fill=none,Violet] (f6) {\Large$\bbf_{5^6}$}
			++(0,2) node[fill=none,Navy] (f18) {\Large$\bbf_{5^{18}}$}
			(f)--(f3)--(f9)
			(f)--(f2)--(f6)--(f18)
			(f3)--(f6)
			(f9)--(f18);	
\end{tikzpicture}
\end{center}
\caption{Selected components of $G(3,53,18)$ colored by weight.}\label{fig:weights}
\end{figure}
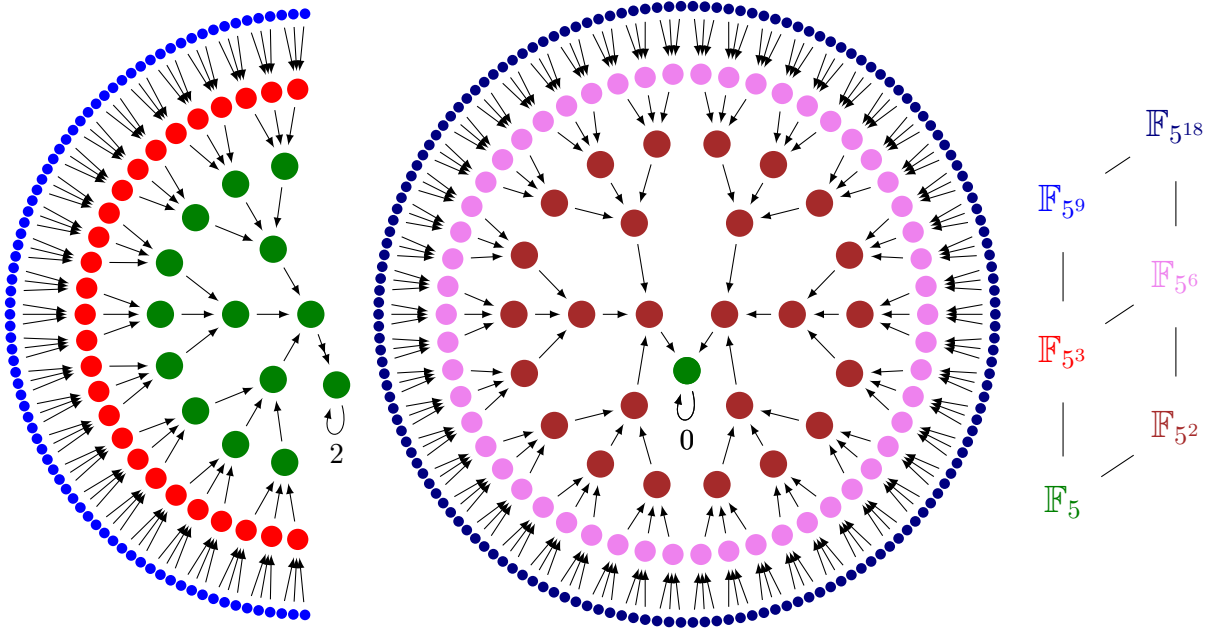

\begin{figure}[!ht]
\begin{tabular}{|c|c|c|c|c|c|}
\hline
\multicolumn{6}{|c|}{$\ell=2, p = 3, n=4$} \\
\hline\hline
Divisors of 80 & Number of points in $\bbf_{3^4}$ & Period & Preperiod & Weight & Cycles of this type\\ \hline
1 & 1 & 1 & 0 & 1 &\multirow{5}{*}1\\
2 & 1 & - & 1 & 1 &\\
4 & 1 & - & 2 & 1 &\\
8 & 2 & - & 3 & 2 &\\
16 & 4 & - & 4 & 4 &\\
\hline
5 & 2 & 2 & 0 & 2 &\multirow{5}{*} 1\\
10 & 2 & - & 1 & 2 &\\
20 & 4 & - & 2 & 4 &\\
40 & 8 & - & 3 & 4 &\\
80 & 16 & - & 4 & 4 &\\
\hline
\hline
Divisors of 82 &&&&&\\
\hline
41 & 20 & 20 & 0 & 4 & \multirow{2}{*}1\\
82 & 20 & - & 1 & 4 &\\
\hline
\end{tabular}

\vspace{3mm} 

\begin{tikzpicture}
[>=latex,grow = up,level distance=10mm, 
every node/.style={circle,outer sep = 1mm}, 
level 1/.style={sibling distance=40mm}, 
level 2/.style={sibling distance=20mm}, 
level 3/.style={sibling distance=10mm},
level 4/.style={sibling distance=5mm,nodes={fill=blue}}]
\node[fill=Green,label = left:2] (2) {}
	child[<-] {node[fill=Green,label=left:$-2$] {}
		child[<<-,level 3/.style={<-,sibling distance=10mm}] {node [fill=Green,label=left:0] {}
			child {node[fill=red]{}
				child{node{}}
				child{node{}}
				}
			child {node[fill=red]{}
				child{node{}}
				child{node{}}
				}
			}
		};
\draw (2) edge[loop below] (2);
\node[fill=red] at(3.25,0) (a){}
	child[<-] {node[fill=red]{}
		child {node[fill=blue]{}
			child {node[fill=blue]{}
				child {node[fill=blue]{}}
				child {node[fill=blue]{}}
				}
			child {node[fill=blue]{}
				child {node[fill=blue]{}}
				child {node[fill=blue]{}}
				}
			}
		child {node[fill=blue]{}
			child {node[fill=blue]{}
				child {node[fill=blue]{}}
				child {node[fill=blue]{}}
				}
			child {node[fill=blue]{}
				child {node[fill=blue]{}}
				child {node[fill=blue]{}}
				}
			}
		};
\node[fill=red] at(7.25,0) (b){}
	child[<-] {node[fill=red]{}
		child {node[fill=blue]{}
			child {node[fill=blue]{}
				child {node[fill=blue]{}}
				child {node[fill=blue]{}}
				}
			child {node[fill=blue]{}
				child {node[fill=blue]{}}
				child {node[fill=blue]{}}
				}
			}
		child {node[fill=blue]{}
			child {node[fill=blue]{}
				child {node[fill=blue]{}}
				child {node[fill=blue]{}}
				}
			child {node[fill=blue]{}
				child {node[fill=blue]{}}
				child {node[fill=blue]{}}
				}
			}
		};
\draw[->] (a) to (b);
\draw[->] (b.south) -- ++(0,-.3) -| (a.south);

\draw (9.6,3.5) rectangle (12.5,.5);
\node[fill = Green,label=right:- weight 1] at (10.1,3) {};
\node[fill = red,label=right:- weight 2] at (10.1,2) {};
\node[fill = blue,label=right:- weight 4] at (10.1,1) {};
\end{tikzpicture}

\vspace{2mm}
\begin{tikzpicture}
[>=latex,scale=.7,grow = up,level distance=14mm, 
every node/.style={circle, fill = blue,outer sep = 1mm}]
\foreach \x in {0,1,...,19}
	\node (\x) at (\x,0) {} child[<-] {node{}};
\foreach \x in {0,1,...,18}
	\draw[->] (\x+.3,0) to (\x+.75,0);
\draw[->] (19.south) -- ++(0,-.5) -| (0.south);
\end{tikzpicture}

\caption{$G(2,3,4)$.}
\label{fig:G(2,3,4)}
\end{figure}

\begin{remark}
The map $x \mapsto x + x\inv$, which provided a key connection between elements of $\bbf_{p^{2n}}^\times$ and $\bbf_{p^n}$, is quite interesting in its own right. For this rational map, one can construct a  graph over the projective space $\bbp^1(\bbf_q)$ in the same way. Ugolini showed that when $\bbf_q$ is a finite field of characteristic 2, 3, or 5, the components of these graphs are also radially symmetric \cite{Ugo12,Ugo13}.
\end{remark}

\section{\large Decomposition of Primes.} 
\subsection{Main result}
We now use the theory of the graphs developed in the previous section to describe how primes not dividing the discriminant of $T_\ell^n(x)-t$ decompose in the Chebyshev radical extensions. If $p$ does not divide the discriminant, the decomposition of $p\bbz$ in $K_n$ is given by the factorization of $T_\ell^n(x)-t$ modulo $p$, which follows immediately from our previous results.

\begin{thm} \label{th:main}
Let $\ell$ and $p$ be distinct primes, and let $t\in \bbz$ such that every iterate $T_\ell^n(x)-t$ is irreducible. Let $\mu, D_1$, and $D_2$ be defined as before, and put $v = \nu_\ell(D_1)$ and $\pper(\,\overline t\,) = \rho$, where $\overline t$ denotes the reduction of $t$ modulo $p$.

\textbf{Case 1:} $\ell$ is odd.
\begin{enumerate}
\item If $\rho>0$, then $T_\ell^n(x)-t$ factors into 
\begin{enumerate}
\item $\ell^n$ factors of degree $1$ if $1 \le n \le v-\rho$, or
\item $\ell^{v-\rho}$ factors of degree $\ell^{n-v+\rho}$ if $n > v-\rho$.
\end{enumerate}
In particular, if $\rho=v$, then $T_\ell^n(x)-t$ is irreducible modulo $p$.

\item If $\overline t \in \{-2,2\}$, then $T_\ell^n(x)-t$ factors into 
\begin{enumerate}
\item one degree $1$ factor, and
\item $\sum_{k=0}^{n-1}\frac{\ell-1}{2\mu}\ell^k$ additional factors of degree $\mu$ if $1 \le n \le v$, and
\item $\frac{\ell-1}{2\mu}\ell^{v-1}$ additional factors of degree $\mu\ell^k$ for each $k \in \{1,2,\ldots, n-v\}$.
\end{enumerate}
Each of the factors from parts (b) and (c) have multiplicity $2$.

\item If $\overline t\not\in\{-2,2\}$ is periodic and $\overline t$ corresponds to a divisor of $D_1$, then $T_\ell^n(x)-t$ factors into
\begin{enumerate}
\item one degree $1$ factor, and
\item $\sum_{k=0}^{n-1}\frac{\ell-1}{\mu}\ell^k$ additional factors of degree $\mu$ if $1 \le n \le v$, and
\item $\frac{\ell-1}{\mu}\ell^{v-1}$ additional factors of degree $\mu\ell^k$ for each $k \in \{1,2,\ldots, n-v\}$.
\end{enumerate}

\item If $\overline t \not\in\{-2,2\}$ is periodic and corresponds to a divisor of $D_2$, then $T_\ell^n(x)-t$ factors into 
\begin{enumerate}
\item one degree $1$ factor, and
\item $\sum_{k=0}^{n-1}\frac{\ell-1}{2\mu}\ell^k$ additional factors of degree $2\mu$ if $1 \le n \le v$, and
\item $\frac{\ell-1}{2\mu}\ell^{v-1}$ additional factors of degree $2\mu\ell^k$ for each $k \in \{1,2,\ldots, n-v\}$.
\end{enumerate}
\end{enumerate}

\textbf{Case 2:} $\ell=2$.
\begin{enumerate}
\item If $\rho>0$ and $\overline t \neq -2$ is connected to a cycle corresponding to a divisor of $D_1$, then $T_2^n(x)-t$ factors into 
\begin{enumerate}
\item $2^n$ factors of degree 1 if $1 \le n \le v-\rho$, or
\item $2^{v-\rho}$ factors of degree $2^{n-v+\rho}$ if $n > v-\rho$.
\end{enumerate}
In particular, if $\rho=v$, then $T_2^n(x)-t$ is irreducible modulo $p$.

\item If $\rho>0$ and $\overline t$ is connected to a cycle corresponding to a divisor of $D_2$, then $T_2^n(x)-t$ factors into 
\begin{enumerate}
\item $2^{n-1}$ factors of degree $2$ if $1 \le n \le v-\rho$, or
\item $2^{v-\rho}$ factors of degree $2^{n-v+\rho}$ if $n > v-\rho$.
\end{enumerate}

\item If $\overline t = -2$, then $T_2^n(x)-t$ factors into 
\begin{enumerate}
\item $2^{n-1}$ factors of degree 1 and multiplicity 2 if $1 \le n \le v-\rho$, or
\item $2^{v-\rho-1}$ factors of degree $2^{n-v+\rho}$ and multiplicity 2 if $n > v-\rho$.
\end{enumerate}

\item If $\overline t = 2$, then $T_2^n(x)-t$ factors as $(x-2)(x+2)(T_2^{n-1}(x)+2)$, which is given by (3).

\item If $\overline t\neq 2$ is periodic and $\overline t$ corresponds to a divisor of $D_1$, then $T_2^n(x)-t$ factors into
\begin{enumerate}
\item one degree $1$ factor, and
\item $\sum_{k=0}^{n-1}2^k$ additional factors of degree 1 if $1 \le n \le v$, and
\item $2^{v-1}$ additional factors of degree $2^k$ for each $k \in \{1,2,\ldots, n-v\}$.
\end{enumerate}

\item If $\overline t\neq 2$ is periodic and $\overline t$ corresponds to a divisor of $D_2$, then $T_2^n(x)-t$ factors into
\begin{enumerate}
\item two degree $1$ factors, and
\item $\sum_{k=0}^{n-2}2^k$ additional factors of degree $2$ if $1 \le n \le v$, and
\item $2^{v-2}$ additional factors of degree $2^{k+1}$ for each $k \in \{1,2,\ldots, n-v\}$.
\end{enumerate}
\end{enumerate}
\end{thm}

\begin{proof}
The statement is complex, but the method of proof is simple. Let $G = G(\ell,p,2\mu\ell^n)$. Certainly, $a$ is a root of $T_\ell^n(x)-t$ modulo $p$ if and only if $G$ contains a path of length $n$ from $a$ to $\overline t$. Moreover, if $a$ is a root of $T_\ell^n(x)-t$ modulo $p$, then the factorization of $T_\ell^n(x)-t$ modulo $p$ contains a factor of degree $\weight(a)$. By Theorem \ref{th:structure of graph}, the graph $G$ contains $\ell^n$ paths of length $n$ that terminate at $\overline t$ provided $\overline t$ is not one of the critical values $\{-2,2\}$. If $\overline t \in \{-2,2\}$, then there are $\ell^n/2$ such paths. In either case, the preperiod of each root is easily determined by simply back-tracking along each of the paths terminating at $\overline t$. For example, if $\overline t \not\in\{-2,2\}$ and $\pper(\,\overline t\,) > 0$, then the paths of length $n$ terminating at $\overline t$ originate at vertices of preperiod $\pper(\,\overline t\,)+n$. If $\overline t \not\in\{-2,2\}$ and $\pper(\,\overline t\,)=0$, then there are $(\ell-1)\ell^{k-1}$ paths of length $n$ originating from vertices of preperiod $k$ for each $k \in \{1, \ldots, n\}$, and one path originating from a vertex of preperiod 0. Once the preperiods of the roots are known, their weights are given by Theorem \ref{th:weights}, and the degrees of the irreducible factors $T_\ell^n(x)-t$ follow immediately. 
\end{proof}

\begin{cor}
Suppose that $\ell$ is an odd prime. Then $T_\ell^n(x)-t$ is reducible modulo $p$ if and only if $T_\ell(x)-t$ has a root modulo $p$.
\end{cor}

\begin{proof}
Clearly, if $T_\ell(x)-t$ has a root modulo $p$, then $T_\ell^n(x)-t = T_\ell(T_\ell^{n-1}(x))-t$ is reducible. On the other hand, Theorem \ref{th:main} Case 1.1 implies that if $T_\ell^n(x)-t$ is reducible modulo $p$, then $\pper(\,\overline t\,) < v$, hence $T_\ell(x)-t$ has a root modulo $p$.
\end{proof}

Note that Theorem \ref{th:intro.1} in the introduction is a rewording of Theorem \ref{th:main}, Case 1.1. This particular case also shows that for a fixed prime $\ell$, there are infinitely many values $t$ such that every iterate $T_\ell^n(x)-t$ is irreducible. By Dirichlet's theorem on arithmetic progressions, there are infinitely many primes $p$ satisfying $p\equiv \pm1 \pmod\ell$. Let $p$ be one of these primes, then by Theorem \ref{th:point count}, the graph $G(\ell,p,1)$ contains $(\ell-1)\ell^{v-1}\omega_m/2$ vertices of maximal preperiod. That is, there are $(\ell-1)\ell^{v-1}\omega_m/2$ equivalence classes modulo $p$ such that if $t$ is in one of these classes, then every iterate $T_\ell^n(x)-t$ is irreducible.

\subsection{Discriminant formula}
We now consider prime divisors of the discriminant of $T_\ell^n(x)-t$. It turns out that the set of primes
\begin{align*}
S_{\ell,t} := \{ p \mid \disc_x(T_\ell^n(x)-t) \colon n\ge 1\}
\end{align*} 
is finite, which is due to the fact that the Chebyshev polynomials are \emph{postcritically finite}.

\begin{defn}
Let $K$ be a number field and let $\phi \in K[x]$. The set of \emph{critical points} of $\phi$ is $\cal R_\phi =\{r \in \overline K \colon \phi'(r) = 0\}$. The polynomial $\phi$ is postcritically finite if all of its critical points are preperiodic.
\end{defn}

\begin{prop} \label{prop:chebipper} 
The Chebyshev polynomial $T_n$ is postcritically finite.
\end{prop}

\begin{proof}The derivative of $T_n$ is
\begin{align*}
\frac{d}{dx}T_n = n U_{n-1},
\end{align*}
where $U_{n-1}$ is the degree $n-1$ Chebyshev polynomial of the second kind and is defined by the relation
\begin{align*}
U_{n-1}(2\cos\theta) = \ds\frac{\sin(n\theta)}{\sin\theta}.
\end{align*}
The roots of $U_{n-1}$ are $x_k = 2\cos\left(\frac{k \pi}{n}\right)$, for $k \in\{1, 2, \ldots, n-1\}$, and it is easily seen that the forward orbit of these values is finite.
\begin{align*}
&T_n\left(2\cos\left(\frac{k\pi}{n}\right)\right) = 2\cos(k\pi) = (-1)^k \, 2,\\
&T_n(-2) = T_n(2\cos(\pi)) = 2\cos(n\pi) = (-1)^n\,2,\\
&T_n(2) = T_n(2\cos(0)) = 2\cos(0) = 2.
\end{align*}
\end{proof}

\begin{defn}
Let $K$ be a number field and $\phi = a_dx^d + \cdots + a_0 \in K[x]$. Let $\cal B(\phi)=\{\phi(r) \colon r \in \cal R_\phi\}$ denote the set of \emph{critical values} of $\phi$. The multiplicities of the critical points of $\phi$ are identified by
\begin{align*}
\phi'(x) = da_d\prod_{r \in \cal R_\phi}(x-r)^{m_r(\phi)},
\end{align*}
where $\cal R_\phi$ is the set of critical points of $\phi$. For each critical value $\beta \in \cal B(\phi)$, define
\begin{align*}
\cal M_\beta(\phi) = \sum_{r \in \cal R_\phi, \, \phi(r) = \beta} m_r(\phi).
\end{align*}
\end{defn}

Since the Chebyshev polynomials are postcritically finite, the critical values of their iterates coincide. In fact, the union $\bigcup_{n\ge 1} \cal B(T_\ell^n(x)-t)$ is a finite set. Using a discriminant formula from \cite{AHM05}, we can show that $S_{\ell,t}$ is finite.

\begin{prop}\label{prop:disc} 
For $\phi = a_dx^d + \cdots + a_0 \in K[x]$ and $n \ge 1$
\begin{align*}
\disc_x(\phi^n(x)-t) = (-1)^{(d^n-1)(d^n-2)/2}d^{nd^n}a_d^{(d^n-1)^2/(d-1)}\prod_{\beta \in \cal B(\phi^n(x)-t)}(t-\beta)^{\cal M_\beta(\phi^n)}.
\end{align*}
\end{prop}

\begin{proof}
\cite{AHM05}, Proposition 3.2.
\end{proof}

\begin{cor} \label{cor:disc} 
\
\begin{enumerate}
\item If $\ell$ is odd, then
\begin{align*}
\disc_x(T_\ell^n(x)-t) = \ell^{n\ell^n}(4-t^2)^{(\ell^n-1)/2},
\end{align*}
and the discriminants satisfy the recursion
\begin{align*}
\disc_x(T_\ell^{n+1}(x)-t) = \ell^{\ell^{n+1}}\big(\disc_x(T_\ell^n(x)-t)\big)^\ell(4-t^2)^{(\ell-1)/2}.
\end{align*}
\item If $\ell=2$, then 
\begin{align*}
\disc_x(T_2^n(x)-t) = 2^{n2^n}(2-t)(4-t^2)^{2^{n-1}-1},
\end{align*}
and the discriminants satisfy the recursion
\begin{align*}
\disc_x(T_2^{n+1}(x)-t) = 4^{2^n}(2-t)\big(\disc_x(T_2^n(x)-t)\big)^2.
\end{align*}
\end{enumerate}
\end{cor}

\begin{proof} We begin with the case where $\ell$ is odd. In Proposition \ref{prop:chebipper}, we identified the $\ell^n-1$ distinct critical points of $T^n_\ell(x)$, and identified 2 and $-2$ as the two critical values. Each critical value is the image of half of the critical points, so $\cal M_2(T_\ell^n) = \cal M_{-2}(T_\ell^n) = (\ell^n-1)/2$ and we can apply Proposition \ref{prop:disc}. For the recursion, 
\begin{align*}
\disc_x(T_\ell^{n+1}(x)-t) &= \ell^{(n+1)\ell^{n+1}}(4-t^2)^{(\ell^{n+1}-1)/2}\\
&=\big(\ell^{n\ell^n}\big)^\ell\ell^{\ell^{n+1}}\big((4-t^2)^{(\ell^n-1)/2}\big)^\ell(4-t^2)^{(\ell-1)/2}\\
&= \big(\disc_x(T_\ell^n(x)-t)\big)^\ell(4-t^2)^{(\ell-1)/2}.
\end{align*}

\sss In the case $\ell=2$, $T_2^n$ has $2^n-1$ critical points. It is a quick check to verify that $\cal M_2(T_2^n) = 2^{n-1}-1$ and $\cal M_{-2}(T_2^n) = 2^{n-1}$, and the discriminant formula follows immediately. For the recursion,
\begin{align*}
\disc_x(T_2^{n+1}(x)-t) &= 2^{(n+1)2^{n+1}}(2-t)(4-t^2)^{2^n-1}\\
&= 4^{2^n}\big(2^{n2^n}\big)^2(2-t)^3\big((4-t^2)^{2^{n-1}-1}\big)^2\\
&=4^{2^n}(2-t)\big(\disc_x(T_2^n(x)-t)\big)^2.
\end{align*}
\end{proof}

\subsection{Reciprocity}
A result of Aitken, Hajir, and Maire reinterprets the classical result of Dedekind regarding the decomposition of primes in terms of the graphs of postcritically finite polynomials over finite fields. We give a specialization of their result, which is tailored to our setting. Recall that $K_n$ denotes a Chebyshev radical extension associated to the polynomial $T_\ell^n(x)-t$, and $G(\ell,p,2\mu\ell^n)$ is the graph of $T_\ell$ over the finite field of order $p^{2\mu\ell^n}$. 

\begin{prop} \label{prop:AHM}
Suppose $T_\ell^n(x)-t$ is irreducible over $\bbz[x]$ for all $n \ge 1$, and the prime $p$ does not divide $\ell(4-t^2)$. Let $\mu = \ord_{(\bbz/\ell\bbz)^\times/(\pm1)}p$. Then for $k \ge 1$, the number of degree $k$ primes of $K_n$ lying over $p\bbz$ is $N/k$, where $N$ is the number of paths of length $n$ in $G(\ell,p,2\mu\ell^n)$ that start at a vertex of weight $k$ and end at $\overline t$.
\end{prop}

\begin{proof}
\cite{AHM05}, Proposition 5.1.
\end{proof}

The following results regarding primes now follow immediately from Proposition \ref{prop:AHM} and Theorem \ref{th:main}.

\begin{cor} \label{cor:prime decomposition}
Suppose $p\nmid \ell(4-t^2)$, and let $\rho = \pper(\,\overline t \,)$.

\textbf{Case 1:} $\ell$ is odd.
\begin{enumerate}
\item If $\rho>0$, then $p$ splits in $K_1, \ldots, K_{v-\rho}$ and is totally inert afterwards.
\item If $\rho=0$, then there is at least one prime of degree 1 above $p$ at every level, and exactly one prime of degree one at the infinite level.
\end{enumerate}

\textbf{Case 2:} $\ell=2$.
\begin{enumerate}
\item If $\rho>0$ and $\overline t$ is connected to a cycle corresponding to a divisor of $D_1$, then $p$ splits completely in $K_1, \ldots, K_{v-\rho}$ and is totally inert afterwards.
\item If $\rho>0$ are $\overline t$ is connected to a cycle corresponding to a divisor of $D_2$, then $p$ is inert in $K_1$, then splits completely in $K_2, \ldots, K_{v-\rho+1}$, and is completely inert afterwards.
\item If $\rho=0$, then there is at least one prime of degree 1 above $p$ at every level, and exactly one prime of degree one at the infinite level.
\end{enumerate}
\end{cor}

In our introduction, we noted two special cases ($\ell$ odd and $t=2$, and $\ell=2$ and $t=0$) for their connection to the cyclotomic fields. We show that our result regarding the decomposition of primes coincides with cyclotomic reciprocity in these cases.

\begin{example} [Cyclotomic $\bbz_2$ extension]
Let $\ell=2$, $t=0$, and $p\neq 2$ a prime. By Theorem \ref{th:structure of graph}, $0 \in G(\ell,p,1)$ has preperiod 2. Hence by Corollary \ref{cor:prime decomposition}, the prime $p$ splits in $K_n = \bbq(\zeta_{2^{n+2}})^+$ if and only if $p \equiv \pm1 \bmod 2^{n+2}$.
\end{example}

\begin{example} [Cyclotomic $\bbz_\ell$ extension]
Let $\ell$ be an odd prime and $t=2$. In this case, we need to consider the minimal polynomial $m_{\ell^n}(x)$ of $\zeta_{\ell^n} + \zeta_{\ell^n}\inv$, which necessarily divides $T_\ell^n(x)-2$. The distribution of the roots given in Theorem \ref{th:structure of graph} shows that $m_{\ell^n}(x)$ splits if and only if $T_\ell^n(x)-2$ splits. Moreover, $\bbq(\zeta_{\ell^n})^+$ is monogenic and its ring of integer is generated by $\zeta_{\ell^n} + \zeta_{\ell^n}\inv$, hence the discriminant of $m_{\ell^n}(x)$ is equal to the discriminant of $\bbq(\zeta_{\ell^n})^+$, which is a power of $\ell$. Thus, for any prime $p\neq \ell$, the decomposition of $p\bbz$ in $\bbq(\zeta_{\ell^n})^+$ is determined by Corollary \ref{cor:prime decomposition}. By Theorem \ref{th:main}, $T_\ell^n(x)-2$ splits modulo $p$ if and only if $p \equiv \pm1 \bmod\ell^n$, thus $p\bbz$ splits in $\bbq(\zeta_{\ell^n})^+$ if and only if $p \equiv \pm1 \bmod\ell^n$.
\end{example}

\begin{example} 
Consider $T_2^n(x)-105$ and the associated extension $K_n$. These polynomials are irreducible since $105 \equiv 0 \pmod 3$, and 0 is a vertex of maximal height in $G(2,3,1)$ (see Figure \ref{fig:G(2,3,1)}). The discriminant formula in Corollary \ref{cor:disc} shows that only 2, 103, and 107 can ramify in the tower of fields. The decomposition of a prime $p$ (different from 2, 103, and 107) in these extensions is determined by the tree rooted at $105 \bmod p \in G(2,p,2^{n+1})$. Noting that 105 is divisible by 3, 5, and 7, the behavior of 3, 5, and 7 is determined by the tree rooted at 0, and thus the splitting of these primes is exactly the same as in the cyclotomic $\bbz_2$-extension. Namely, 3 and 5 are inert, and 7 splits in $K_1$.

For $p=11$, we see that $105 \equiv -5 \pmod{11}$, and therefore we must consider the tree rooted at $-5$. However, 0 and $-5$ occupy the same position in the graph $G(2,11,1)$ (see Figure \ref{fig:G(2,11,1)}), hence the tree rooted at $-5$ in $G(2,11,2^{n+1})$ is identical to the tree rooted at 0. Thus the splitting behavior of 11 in this extension is also identical to its behavior in the $\bbz_2$-extension, that is, it is inert.

The prime 13 is the smallest prime that displays a different behavior in this tower of fields than in the cyclotomic extension. Since $105 \equiv 1 \pmod{13}$, we focus on the tree rooted at $1$ in $G(2,13,2^{n+1})$. We can see in $G(2,13,1)$ (see Figure \ref{fig:G(2,13,1)}) the vertex 1 is contained in a tree of maximal height and is the root of a tree of height one. By the previous theorem, 13 splits in $K_1$, and is otherwise inert.
\end{example}

\begin{figure}[!ht]
\begin{tikzpicture}[>=latex, grow=up, every node/.style={circle,fill=Green},scale=.6,
level 2/.style={sibling distance=10mm}]
\node[label = left:2] at (0,0) (2) {}
	child[<-] { node[label=left:$-2$] {}
		child[<<-] {node[label=left:0] {} }
		};
\end{tikzpicture}
\caption{$G(2,3,1)$}
\label{fig:G(2,3,1)}
\end{figure}
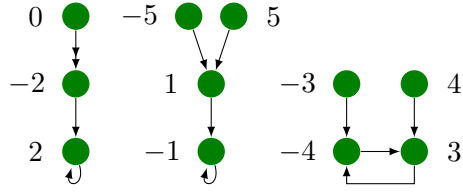
\begin{figure}[!ht]
\begin{tikzpicture}[>=latex, grow=up, every node/.style={circle,fill=Green},scale=.6,
level 2/.style={sibling distance=10mm}]
\node[label = left:2] at (0,0) (2) {}
	child[<-] { node[label=left:$-2$] {}
		child[<<-] {node[label=left:0] {} }
		};
\draw (2) edge[loop below] (2);
\node[label=left:$-1$] at (3,0) (-1) {}
	child[<-] { node[label=left:1] {}
		child { node[label=right:5] {} }
		child { node[label=left:$-5$] {} }
		};
\draw (-1) edge[loop below] (-1);
\node[label=left:$-4$] at (6,0) (-4) {}
	child[<-] { node[label=left:$-3$] {} };
\node[label=right:3] at (7.5,0) (3) {}
	child[<-] { node[label=right:4] {} };
\draw[->] (-4) edge (3);
\draw[->] (3.south) -- ++(0,-.4) -| (-4.south);
\end{tikzpicture}
\caption{$G(2,11,1)$}
\label{fig:G(2,11,1)}
\end{figure}
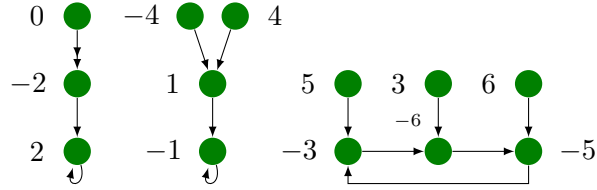
\begin{figure}[!ht]
\begin{tikzpicture}[>=latex, grow=up, every node/.style={circle,fill=Green},scale=.6,
level 2/.style={sibling distance=10mm}]
\node[label = left:2] at (0,0) (2) {}
	child[<-] { node[label=left:$-2$] {}
		child[<<-] {node[label=left:0] {} }
		};
\draw (2) edge[loop below] (2);
\node[label=left:$-1$] at (3,0) (-1) {}
	child[<-] { node[label=left:1] {}
		child { node[label=right:4] {} }
		child { node[label=left:$-4$] {} }
		};
\draw (-1) edge[loop below] (-1);
\node[label=left:$-3$] at (6,0) (-3) {}
	child[<-] { node[label=left:5] {} };
\node[label=above left:\tiny $-6$] at (8,0) (-6) {}
	child[<-] { node[label=left:3] {} };
\node[label=right:$-5$] at (10,0) (-5) {} 
	child[<-] { node[label=left:6] {} };
\draw[->] (-3) edge (-6) (-6) edge (-5);
\draw[->] (-5.south) -- ++(0,-.4) -| (-3.south);
\end{tikzpicture}
\caption{$G(2,13,1)$}
\label{fig:G(2,13,1)}
\end{figure}

\section{\large Density of Periodic Points.}
Jones \cite{Jon12} proposed the question: what proportion of points in $G(\ell,p,n)$ are periodic? In particular, does the following limit exist?
\begin{align*}
\lim_{n\to \infty} \frac{\#\{\text{periodic points in }G(\ell,p,n)\}}{p^n}
\end{align*}
By Theorem \ref{th:point count}, we see that
\begin{align*}
\lim_{n\to \infty} &\frac{\#\{\text{periodic points in }G(\ell,p,n)\}}{p^n} \\
&= \lim_{n \to \infty} \frac{\omega^-+\omega^+}{2p^n}\\
&= \lim_{n\to\infty} \frac{1}{2p^n}\left(\frac{p^n-1}{\ell^{\lambda^-}} + \frac{p^n+1}{\ell^{\lambda^+}}\right) \\
&= \lim_{n\to\infty}\frac{1}{2}\left(\frac{1}{\ell^{\lambda^-}} + \frac{1}{\ell^{\lambda^+}}\right) \\
&= 
\begin{cases} 
\ds\lim_{n\to\infty}\frac{1}{2\ell^{\lambda_m}} + \frac{1}{2} & \text{ if $\ell\neq 2$}, \\ 
\ds\lim_{n\to\infty}\frac{1}{2^{\lambda_m+1}} + \frac{1}{4} & \text{ if $\ell=2$}.
\end{cases}
\end{align*}

Recalling Lemma \ref{lem:valuation}, we see that the limits, as written, do not exist since 
for infinitely many positive integers $n$
\begin{align*}
\lambda_m = \nu_\ell(p^{2n}-1) = 
\begin{cases}
0 & \text{ if $\ell$ is odd,}\\
3 & \text{ if $\ell=2$.}
\end{cases}
\end{align*}

This limit, however, can exist if we restrict ourselves to an appropriate tower of fields. For example, consider the sequence
\begin{align*}
\big\{a_n\big\}_{n=1}^\infty = \big\{2, \quad 2^2\,3, \quad 2^3 \,3^2 \,5, \quad 2^4 \,3^3 \,5^2\,7 , \quad \ldots \quad , \quad2^n \,3^{n-1} \,5^{n-2} \,7^{n-3} \,\cdots \,p_n, \quad \ldots \big\},
\end{align*}
where $p_n$ denote the $n$-th prime, and the corresponding tower of fields
\begin{align*}
\bbf_p \subset \bbf_{p^{a_1}} \subset \bbf_{p^{a_2}} \subset \bbf_{p^{a_3}} \subset \cdots \subset \overline\bbf_p,
\end{align*}
where $\bigcup \bbf_{p^{a_n}} = \overline \bbf_p$. This sequence of $a_n$'s is constructed so that for any prime $\ell$, $c$ divides all but a finitely many $a_n$, thus $\nu_\ell(a_n)$ (and therefore $\lambda_m$) grows monotonically and without bound as $n$ increases. Taking the limit up this tower of fields, 
\begin{enumerate}
\item if $\ell$ is odd,
\begin{align*}
\lim_{n\to \infty} \frac{\#\{\text{periodic points in }G(\ell,p,a_n)\}}{p^{a_n}} = \ds\lim_{n\to\infty}\frac{1}{2\ell^{\nu_\ell(p^{2c}-1)+\nu_\ell(a_n)}} + \frac{1}{2} = \frac{1}{2},
\end{align*}
\item if $\ell = 2$, 
\begin{align*}
\lim_{n\to \infty} \frac{\#\{\text{periodic points in }G(2,p,a_n)\}}{p^{a_n}} = \ds\lim_{n\to\infty}\frac{1}{2^{\nu_2(p^2-1)+\nu_2(a_n)}} + \frac{1}{4} = \frac{1}{4}.
\end{align*}
\end{enumerate}

\section{Acknowledgements}
The author is exceptionally grateful to Farshid Hajir for his guidance and insights shared over the course of many conversations. The author also thanks Jeff Hatley for proofreading this work, as well as the anonymous referees for their extensive comments and for bringing the rich history of permutation polynomials to the attention of the author.

\bibliographystyle{plain}
\bibliography{ChebyshevActionV2}
\end{document}